\documentclass[11pt,a4paper]{amsart}
\usepackage{enumerate}
\usepackage{amsfonts,amssymb}
\usepackage{amsthm,amsmath,amstext,xfrac}
 \usepackage{tikz}
 \usetikzlibrary{matrix,arrows,calc,decorations.pathreplacing,decorations.markings}
 \usetikzlibrary{decorations.pathmorphing}
 \usetikzlibrary{positioning}

\usepackage{cancel,xspace}
\usepackage{epsfig,fancybox}
\usepackage{graphicx,mathrsfs,ifpdf}
\usepackage{setspace}
\usepackage{cleveref}
\usepackage{enumerate}
\usepackage[colorinlistoftodos,prependcaption,textsize=tiny]{todonotes}

\newtheorem{theorem}{Theorem}
\newtheorem{lemma}[theorem]{Lemma}

\theoremstyle{definition}

\newtheorem*{acknowledgement}{Acknowledgments}

\numberwithin{equation}{section}
\newcommand{\RR}{\mathbb R}

\newcommand{\claimproofend}{\hspace*{.1mm}\hspace{\fill}}

\begin{document}
\title{Bipartite complements of circle graphs}
\author{Louis Esperet}
\thanks{}
\address{Laboratoire G-SCOP (CNRS, Univ. Grenoble Alpes),
  Grenoble, France}
\email{louis.esperet@grenoble-inp.fr}
\author{Mat\v ej Stehl\'ik} \thanks{Partially supported by ANR Projects GATO
(\textsc{anr-16-ce40-0009-01}) and GrR (\textsc{anr-18-ce40-0032})} \address{Laboratoire G-SCOP, Univ.\ Grenoble Alpes, France}
\email{matej.stehlik@grenoble-inp.fr} \date{\today}

\begin{abstract}
Using an algebraic characterization of circle graphs, Bouchet proved
in 1999
that if a bipartite graph $G$ is the complement of a circle graph,
then $G$
is a circle graph. We give an elementary proof of
this result.
\end{abstract}
\maketitle

A graph is a \emph{circle graph} if it is the intersection graph of the
chords of a circle. Using an algebraic characterization of circle
graphs proved by Naji~\cite{Naj85} (as the class of graphs satisfying a certain
system of equalities over $\mathrm{GF}(2)$), Bouchet proved
the following result in~\cite{Bou99}.

\begin{theorem}[Bouchet~\cite{Bou99}]\label{th:bouchet}
If a bipartite graph $G$ is the complement of a circle graph,
then $G$
is a circle graph.
\end{theorem}

The known proofs of Naji's theorem are fairly involved~\cite{Gas97,GL18,Naj85,Tral17}, and
Bouchet~\cite{Bou99} (see also~\cite{DGS14}) asked whether, on the
other hand, Theorem~\ref{th:bouchet}  has an elementary proof. The purpose of this short note is to
present such a proof.

\medskip

We will need two simple lemmas.
Given a finite set of points $X \subset \RR^2$ of even cardinality, a line $\ell$
\emph{bisects} the set $X$ if each open half-plane defined by $\ell$ contains precisely
$|X|/2$ points.
The following lemma is an immediate consequence of the $2$-dimensional \emph{discrete ham sandwich
theorem} (see e.g.~\cite[Corollary~3.1.3]{Mat03}), and is equivalent to the \emph{necklace splitting problem} with two types of beads. In order to keep this note self-contained, we include a short proof.

\begin{lemma}\label{lem:hs-circle}
Let $X,Y \subset \RR^2$ be disjoint finite point sets of even cardinality on a circle $C$.
Then there exists a line $\ell$ simultaneously bisecting both $X$ and $Y$.
\end{lemma}

\begin{proof}
Let $p_0, \ldots, p_{2n-1}$ be the points of $X \cup Y$ in cyclic
order along $C$. For $0\le i \le 2n-1$ we denote by $I_i$ the set
$\{p_i,p_{i+1} \ldots, p_{i+n-1}\}$ (here and in the remainder of the
proof, all indices are considered modulo $2n$).
Clearly, for every $0\le i \le n-1$, there exists a line $\ell_i$ in $\RR^2$
bisecting the points of $X \cup Y$, with $I_i$ on one side of $\ell_i$
and $I_{i+n}$ on the other side. For $0\le i \le 2n-1$,  define
$f(i)=|X \cap I_i|-\tfrac12|X|$. Note that since $X$ has even
cardinality, each $f(i)$ is an integer.

To prove the lemma, it suffices to show that $f(i)=0$ for some $0\le i
\le n-1$,
for then $|X \cap I_i|=\tfrac12|X|$ and $|Y \cap I_i|=\tfrac12(|X|+|Y|)-|X \cap I_i|=\tfrac12|Y|$.
If $f(0)=0$ then we are done, so let us assume that $f(0) \neq
0$. 
Without loss of generality $f(0)<0$, and hence $f(n)=-f(0)>0$.
Since $f(i+1)-f(i)\in \{-1,0,1\}$ for all $0\le i \le n-1$, there exists
$1\le i \le  n-1$ such that $f(i)=0$, as required.
\end{proof}

\begin{lemma}\label{lem:cho}
Consider a set of pairwise intersecting chords $c_1,\ldots,c_n$ of a circle $C$, with
pairwise distinct endpoints. Then any line $\ell$ that bisects
the $2n$ endpoints of the chords intersects all the chords $c_1,\ldots,c_n$.
\end{lemma}

\begin{proof}
Assume for the sake of contradiction that some chord $c_i$ does not
intersect $\ell$. Then $c_i$ lies in one of the two open half-planes
defined by $\ell$, say to the left of $\ell$. Since $\ell$ bisects the $2n$ endpoints of the
chords, it follows that there
is another chord $c_j$ that does not intersect $\ell$ and which lies
in the half-plane
to the right of $\ell$. This implies that $c_i$ and $c_j$ do not
intersect, which is a contradiction.
\end{proof}

We are now ready to prove Theorem~\ref{th:bouchet}.

\medskip

\noindent {\it Proof of Theorem~\ref{th:bouchet}.} Consider a bipartite
graph $G$ such that its complement $\overline{G}$ is a circle
graph. In particular, for any vertex $v_i$ of $\overline{G}$ there is
a chord $c_i$ of some circle $C$ such that any two vertices $v_i$ and
$v_j$ are adjacent in $\overline{G}$ (equivalently, non-adjacent in
$G$) if and only if the chords $c_i$ and $c_j$ intersect.  Since $G$ is bipartite, the
vertices $v_1,\ldots,v_n$ (and the corresponding chords
$c_1,\ldots,c_n$) can be colored with colors red and blue such that any two
chords of the same color intersect. We can assume without loss of
generality that the endpoints of the $n$ chords are pairwise
distinct, so the coloring of the chords also gives a coloring of the
$2n$ endpoints with colors red or blue (with an even number of blue
endpoints and an even number of red endpoints). Since the $2n$
endpoints lie on the circle $C$, it
follows from Lemma~\ref{lem:hs-circle} that there exists a line $\ell$ simultaneously
bisecting the set of blue endpoints and the set of red endpoints.

On one side of $\ell$, reverse the order of the
endpoints of the chords $c_1,\ldots,c_n$ along the circle $C$.
Observe that crossing chords intersecting $\ell$ become non-crossing, and vice versa.
By Lemma~\ref{lem:cho}, $\ell$ intersects all the
chords $c_1,\ldots,c_n$, and thus the resulting circle graph is
precisely $G$. It follows that $G$ is a
circle graph, as desired.
\hfill $\Box$

\begin{acknowledgement}
The authors would like to thank Andr\'as Seb\H o for his remarks on an
early version of the draft.
\end{acknowledgement}

\end{document}